\documentclass[a4paper, intlimits, reqno]{amsart}

\usepackage[english]{babel}
\usepackage[T1]{fontenc}
\usepackage[utf8]{inputenc}

\usepackage{amsmath}
\usepackage{amssymb}
\usepackage{MnSymbol}
\usepackage{amsthm}
\allowdisplaybreaks
\usepackage{amsfonts}
\usepackage{mathrsfs} 
\usepackage{mathtools}
\usepackage{nicefrac}
\usepackage{enumitem}

\usepackage{multicol}
\usepackage{url}
\usepackage{dsfont}
\usepackage[numbers,sort&compress]{natbib}
\usepackage{doi}
\usepackage{prettyref}
\usepackage{xcolor}
\usepackage{tikz}
\usepackage{orcidlink}

\usepackage[textsize=footnotesize]{todonotes}

\newrefformat{defi}{Definition \ref{#1}}
\newrefformat{rem}{Remark \ref{#1}}
\newrefformat{sec}{Section \ref{#1}}
\newrefformat{subsec}{Subsection \ref{#1}}
\newrefformat{prop}{Proposition \ref{#1}}
\newrefformat{thm}{Theorem \ref{#1}}
\newrefformat{cor}{Corollary \ref{#1}}
\newrefformat{ex}{Example \ref{#1}}
\newrefformat{ass}{Assumption \ref{#1}}

\newtheorem{lemma}{Lemma}[section]
\newtheorem{proposition}[lemma]{Proposition}
\newtheorem{theorem}[lemma]{Theorem}

\theoremstyle{definition}
\newtheorem{definition}[lemma]{Definition}
\newtheorem{remark}[lemma]{Remark}

\newcommand{\N}{\mathds{N}}

\newcommand{\R}{\mathds{R}}
\newcommand{\C}{\mathds{C}}
\newcommand{\K}{\mathds{K}}
\newcommand{\1}{\mathds{1}}

\newcommand{\e}{\mathrm{e}}

\newcommand{\topo}{\tau}
\newcommand{\co}{\mathrm{co}}

\newcommand{\calB}{\mathcal{B}}

\newcommand{\calF}{\mathcal{F}}
\newcommand{\calL}{\mathcal{L}}
\newcommand{\calP}{\mathcal{P}}
\newcommand{\calZ}{\mathcal{Z}}

\providecommand{\abs}[1]{\left\lvert#1\right\rvert}
\providecommand{\norm}[1]{\left\lVert#1\right\rVert}
\newcommand{\dupa}[2]{\left\langle #1, #2\right\rangle}
\providecommand{\differential}{\mathrm{d}}
\newcommand{\drm}{\mathrm{d}}
\newcommand{\euler}{\mathrm{e}}

\DeclareMathOperator{\TextRe}{Re}

\renewcommand{\d}{\differential}
\renewcommand{\Re}{\TextRe}

\newcommand{\from}{\colon}
\DeclareMathOperator{\spt}{spt}
\DeclareMathOperator*{\esssup}{ess\,sup}

\newcommand\rlim{
\mathchoice{\vcenter{\hbox{${\scriptstyle{+}}$}}}
{\vcenter{\hbox{$\scriptstyle{+}$}}}
{\vcenter{\hbox{$\scriptscriptstyle{+}$}}}
{\vcenter{\hbox{$\scriptscriptstyle{+}$}}}}
\DeclareMathAlphabet{\mathcal}{OMS}{cmsy}{m}{n}
\newcommand{\vertiii}[1]{{\left\vert\kern-0.25ex\left\vert\kern-0.25ex\left\vert #1 
    \right\vert\kern-0.25ex\right\vert\kern-0.25ex\right\vert}}

\begin{document}

\title[Observability and null-controllability for bi-continuous semigroups]{Final state observability estimates and cost-uniform approximate null-controllability for bi-continuous semigroups}
\author[K.~Kruse]{Karsten Kruse\,\orcidlink{0000-0003-1864-4915}}
\address[Karsten Kruse, Christian Seifert]{Hamburg University of Technology\\ Institute of Mathematics \\ Am Schwarzenberg-Campus~3 \\
21073 Hamburg \\
Germany}
\email{karsten.kruse@tuhh.de}
\author[C.~Seifert]{Christian Seifert\,\orcidlink{0000-0001-9182-8687}}
\email{christian.seifert@tuhh.de}



\date{\today}
\begin{abstract}
    We consider final state observability estimates for bi-continuous semigroups on Banach spaces, i.e.\ for every initial value, estimating the state at a final time $T>0$ by taking into account the orbit of the initial value under the semigroup for $t\in [0,T]$, measured in a suitable norm. We state a sufficient criterion based on an uncertainty relation and a dissipation estimate and provide two examples of bi-continuous semigroups which share a final state observability estimate, namely the Gau{\ss}--Weierstra{\ss} semigroup and the Ornstein--Uhlenbeck semigroup on the space of bounded continuous functions.
    Moreover, we generalise the duality between cost-uniform approximate null-controllability and final state observability estimates to the setting of locally convex spaces for the case of bounded and continuous control functions, which seems to be new even for the case of Banach spaces.
\end{abstract}
\maketitle

\section{Introduction}
\label{sec:intro}

Let $X$ be a Banach space and $(S_t)_{t\geq 0}$ a semigroup on $X$, i.e.\ $S_0 = I$ and $S_{t+s} = S_t S_s$ for all $t,s\geq 0$. Moreover, let $Y$ be another Banach space and $C\in \calL(X;Y)$, a so-called observation operator, and $T>0$. Then $(S_t)_{t\geq0}$ satisfies a \emph{final state observability estimate} w.r.t.\ some (Banach) space $\calZ$ of functions on $[0,T]$ with values in $Y$, if there exists $C_{\mathrm{obs}}\geq 0$ such that
\[\norm{S_T x}_X \leq C_{\mathrm{obs}} \norm{CS_{(\cdot)} x}_{\calZ} \qquad(x\in X).\]
Put differently, we can estimate the norm of the final state $S_Tx$ by just taking into account the observations $CS_tx$ for $t\in [0,T]$. Typical applications stem from evolution equations on some function space over (a subset of) $\R^d$, where $C$ is a restriction operator to a suitable subset $\Omega$ of $\R^d$ (or of the subset of $\R^d$ the functions are defined on) such that we want to control the final state on all of $\R^d$ by just measuring the evolution on the subset $\Omega$. Final state observability estimates have been studied in various contexts due to its relation to null-controllability, see e.g.\ \cite{Carja-88,LebeauR-95,Vieru-05,YuLC-06,Miller-10,NakicTTV-20,GallaunST-20,BGST2022,GMS2022,BGST2023} and references therein.

Classically, the space $\calZ$ is some $L_r$-space with $r\in [1,\infty]$ (when working in Hilbert spaces, one usually chooses $r=2$), and then the final state observability estimate yields the form
\[\norm{S_T x}_X \leq \begin{cases}
    C_{\mathrm{obs}} \left(\int_0^T \norm{CS_t x}_Y^r \drm t\right)^{1/r} & \text{if } r\in [1,\infty),\\
    C_{\mathrm{obs}}\esssup_{t\in [0,T]} \norm{CS_t x}_Y  & \text{if } r=\infty, 
    \end{cases} \qquad(x\in X).\]
Clearly, in order to formulate this final state observability estimate (i.e.\ to have a well-defined right-hand side) we need some regularity of the semigroup. Indeed, we require measurability of $t\mapsto \norm{CS_t x}_Y$ for all $x\in X$. Of course, strong continuity of $(S_t)_{t\geq0}$ yields continuity of these maps and is therefore sufficient, but also weaker regularities are suitable. In \cite{BGST2022}, dual semigroups of strongly continuous semigroups were considered which yield sufficient regularity.

In this short note we aim at two types of results. First, we consider final state observability estimates for so-called bi-continuous semigroups, see e.g.\ \cite{kuehnemund2001,kuehnemund2003}, which are not strongly continuous for the norm-topology on $X$ but only for a weaker topology. Note that dual semigroups are a special case of bi-continuous ones when considering the weak$^*$ topology. 
There are classical examples of bi-continuous semigroups such as the Gau{\ss}--Weierstra{\ss} semigroup on $C_b(\R^d)$, the space of bounded continuous functions (on $\R^d$), as well as the Ornstein--Uhlenbeck semigroup on $C_b(\R^d)$.
Second, we relate cost-uniform approximate null-controllability of a control system sharing only weak continuity properties such as bi-continuity with a final state observability estimate for the dual system, thus generalising the well-known duality in Hilbert and Banach spaces \cite{Douglas-66,Carja-88,Vieru-05,YuLC-06}. Since this demands to work in Hausdorff locally convex spaces, we here focus on continuous control functions which results in the space $\calZ$ above being a space of vector measures.

The paper is organised as follows. In \prettyref{sec:bi-continuous} we review bi-continuous semigroups and then turn to final-state observability estimates in \prettyref{sec:obs} together with two examples in \prettyref{sec:application}.
Final state observability estimates are then related with cost-uniform approximate null-controllability via duality, which we will exploit in our context in \prettyref{sec:Null-controllability}.

\section{Bi-Continuous Semigroups}
\label{sec:bi-continuous}

In this short section we recall some notation and definitions from the theory of bi-continuous semigroups that we need in 
subsequent sections. For a vector space $X$ over the field $\R$ or $\C$ with a Hausdorff locally convex topology $\topo_X$ 
we denote by $(X,\topo_X)'$ the topological linear dual space and just write $X':=(X,\topo_X)'$ 
if $(X,\topo_X)$ is a Banach space.
We use the symbol $\calL(X;Y):=\calL((X,\norm{\cdot}_{X});(Y,\norm{\cdot}_{Y}))$ 
for the space of continuous linear operators from a Banach space $(X,\norm{\cdot}_{X})$ 
to a Banach space $(Y,\norm{\cdot}_{Y})$ and denote by $\norm{\cdot}_{\mathcal{L}(X;Y)}$ the operator norm on 
$\calL(X;Y)$. If $X=Y$, we set $\calL(X):=\calL(X;X)$.

\begin{definition}[{\cite[I.3.2 Definition]{cooper1978}}]\label{defi:saks_mixed}
    Let $(X,\norm{\cdot}_X)$ be a normed space and $\topo_X$ a Hausdorff locally convex topology on $X$.
 \begin{enumerate}
    \item The triple $(X,\norm{\cdot}_X,\topo_{X})$ is called a \emph{Saks space} if $(X,\norm{\cdot}_X)$ is a Banach space, 
    $\topo_{X}\subseteq\topo_{\norm{\cdot}_X}$ and $(X,\topo_{X})'$ is norming for $X$ where 
    $\topo_{\norm{\cdot}_X}$ denotes the $\norm{\cdot}_X$-topology.
    \item The \emph{mixed topology} $\gamma_X\coloneq\gamma(\norm{\cdot}_X,\topo_X)$ is
	the finest linear topology on $X$ that coincides with $\topo_X$ on $\norm{\cdot}_X$-bounded sets and such that 
	$\topo_X\subseteq\gamma_X \subseteq\topo_{\norm{\cdot}_X}$.
\end{enumerate}
\end{definition}

The mixed topology is actually Hausdorff locally convex and the definition given above is equivalent to the one 
from the literature \cite[Section 2.1]{wiweger1961} by \cite[Lemmas 2.2.1, 2.2.2]{wiweger1961}. 

\begin{definition} 
We call a Saks space $(X,\norm{\cdot}_X,\topo_{X})$ \emph{sequentially complete} if $(X,\gamma_X)$ is sequentially complete.
\end{definition}

Due to \cite[Corollary 2.3.2]{wiweger1961} a Saks space $(X,\norm{\cdot}_X,\topo_{X})$ is sequentially complete 
if and only if $(X,\topo_{X})$ is sequentially complete on $\norm{\cdot}_X$-bounded sets, i.e.~every $\norm{\cdot}_X$-bounded 
$\topo_{X}$-Cauchy sequence converges in $X$. In combination with \cite[Remark 2.3 (c)]{kruse_schwenninger2022} 
this yields that a triple $(X,\norm{\cdot}_X,\topo_{X})$ fulfils \cite[Assumptions 1]{kuehnemund2003} 
if and only if it is a sequentially complete Saks space.

\begin{definition}[{\cite[Definition 3]{kuehnemund2003}}]
  \label{defi:bicontinuoussg}
  Let $(X,\norm{\cdot}_X,\topo_{X})$ be a sequentially complete Saks space.
  Let $(S_t)_{t\geq 0}$ in $\calL(X)$ be a semigroup on $X$. We say that $(S_t)_{t\geq 0}$ is
  \emph{(locally) $\topo_X$-bi-continuous} if
  \begin{enumerate}
   \item it is exponentially bounded, i.e.~there exist $M\geq 1$ and $\omega\in\R$ 
   such that $\norm{S_t}_{\calL(X)}\leq M\euler^{\omega t}$ for all $t\geq 0$,
    \item
      $(S_t)_{t\geq 0}$ is a $C_0$-semigroup on $(X, \topo_X)$, i.e.\ for all $x\in X$ the map $[0,\infty)\ni t\mapsto T_tx\in (X,\topo_X)$ is continuous,
    \item it is (locally) bi-equicontinuous, i.e.~for every $(x_n)_{n\in\N}$ in $X$ and $x\in X$ with $\sup\limits_{n\in\N}\norm{x_n}_X<\infty$ and $\topo_X\text{-}\lim\limits_{n\to\infty} x_n = x$ we have
     \[\topo_X\text{-}\lim_{n\to\infty} S_t (x_n-x) = 0\]
     (locally) uniformly for $t\in [0,\infty)$.
  \end{enumerate}
\end{definition}

As in the case of $C_0$-semigroups on Banach spaces we can define generators for bi-continuous semigroups.

\begin{definition}[{\cite[Definition 1.2.6]{farkas2003}}]
Let $(X,\norm{\cdot}_X,\topo_{X})$ be a sequentially complete Saks space
and $(S_t)_{t\geq 0}$ a locally $\topo_{X}$-bi-continuous semigroup on $X$. The \emph{generator} $(-A,D(A))$ 
is defined by
\begin{align*}
D(A)\coloneqq &\Bigl\{x\in X\;|\;\topo_{X}\text{-}\lim_{t\to 0\rlim}\frac{S_{t}x-x}{t}\;\text{exists in } X\;
\text{and }\sup_{t\in(0,1]}\frac{\norm{S_{t}x-x}_{X}}{t}<\infty\Bigr\},\\
-Ax\coloneqq &\topo_{X}\text{-}\lim_{t\to 0\rlim}\frac{S_{t}x-x}{t}\quad (x\in D(A)).
\end{align*}
\end{definition}

\begin{remark}
There is no common agreement whether to use the here presented definition of a generator or its negative. 
Throughout the entire paper we will stick to the definition made above, i.e.~$-A$ is the generator. 
\end{remark}

\section{Final State Observability Estimates for Bi-Continuous Semigroups}
\label{sec:obs}

The final state observability estimate rests on the following abstract theorem. It provides a sufficient criterion stating that an abstract uncertainty principle (also called spectral inequality), see \eqref{eq:UP_BS}, together with a dissipation estimate, see \eqref{eq:DISS_BS}, yields a final state observability estimate, and has its roots in \cite{LebeauR-95}, see also \cite{Miller-10,NakicTTV-20,GallaunST-20,BGST2022}.

\begin{theorem}[{\cite[Theorem A.1]{BGST2022}}] \label{thm:spectral+diss-obs} 
Let $X$ and $Y$ be Banach spaces, $C\in\calL(X;Y)$, $(S_t)_{t\geq 0}$ a semigroup on $X$, $M \geq 1$ and $\omega \in \R$ such that $\norm{S_t}_{\calL(X)} \leq M \euler^{\omega t}$ for all $t \geq 0$, 
and assume that for all $x\in X$ the map $t\mapsto \norm{C S_t x}_Y$ is measurable.
Further, let $\lambda^*\geq 0$, $(P_\lambda)_{\lambda>\lambda^*}$ in $\calL(X)$,
$r \in [1,\infty]$, $d_0,d_1,d_3,\gamma_1,\gamma_2,\gamma_3,T > 0$ with $\gamma_1 < \gamma_2$, and $d_2\geq 1$, and assume that
\begin{equation} \label{eq:UP_BS}
\forall x\in X \ \forall \lambda > \lambda^* \colon \quad 
\norm{P_\lambda x}_{ X } \le d_0 \euler^{d_1 \lambda^{\gamma_1}} \norm{C  P_\lambda x}_{Y }
\tag{UP}
\end{equation}
and
\begin{equation}\label{eq:DISS_BS}
\forall x\in X \ \forall \lambda > \lambda^* \ \forall t\in (0,T/2] \colon \quad 
\norm{(I-P_\lambda) S_t x}_{X} \le d_2 \euler^{-d_3 \lambda^{\gamma_2} t^{\gamma_3}} \norm{x}_{X}.
\tag{DISS}
\end{equation}
Then there exists $C_{\mathrm{obs}}\geq 0$ such that for all $x \in X$ we have
\[
 \norm{S_T x}_{X} \leq \begin{cases}
    C_{\mathrm{obs}} \left(\int_0^T \norm{CS_t x}_Y^r \drm t\right)^{1/r} & \text{if } r\in [1,\infty),\\
    C_{\mathrm{obs}}\esssup_{t\in [0,T]} \norm{CS_t x}_Y & \text{if } r=\infty .
 \end{cases} 
\]
\end{theorem}

\begin{remark}
\label{rem:constant}
The constant $C_{\mathrm{obs}}$ is explicit in all parameters and of the form
\begin{align*}
 C_{\mathrm{obs}} = \frac{C_1}{T^{1/r}} \exp \left(\frac{C_2}{T^{\frac{\gamma_1 \gamma_3}{\gamma_2 - \gamma_1}}} +   C_3 T\right), 
\end{align*}
with $T^{1/r} = 1$ if $r=\infty$, and suitable constants $C_1, C_2, C_3\geq 0$ depending on the parameters; see \cite[Theorem A.1]{BGST2022} as well as \cite[Theorem 2.1]{GMS2022}.
%
\end{remark}

In order to obtain a version of this theorem for bi-continuous semigroups, we need to argue on measurability of $[0,\infty)\ni t\mapsto\norm{CS_t x}_Y$ for all $x\in X$.

\begin{lemma}
\label{lem:measurability}
Let $(X,\norm{\cdot}_X,\topo_X)$ be a sequentially complete Saks space, $(S_t)_{t\geq 0}$ a locally $\topo_X$-bi-continuous semigroup on $X$, $(Y,\norm{\cdot}_Y,\topo_{Y})$ a Saks space, $C\colon X\to Y$ linear and sequentially $\topo_X$-$\topo_Y$-continuous 
on $\norm{\cdot}_X$-bounded sets, and $x\in X$. Then $[0,\infty)\ni t\mapsto\norm{CS_t x}_Y$ is measurable.
\end{lemma}

\begin{proof}
    Since $[0,\infty) \ni t\mapsto \abs{\dupa{y'}{CS_tx}}$ is continuous for all $y'\in (Y,\topo_Y)'$ 
    by the assumptions on $C$ and the exponential boundedness of $(S_t)_{t\geq 0}$, and
    $\norm{CS_t x}_Y = \sup \{ \abs{\dupa{y'}{CS_tx}}\ | \ y'\in (Y,\topo_Y)',\, \norm{y'}_{Y'}\leq 1\}$ for all $t\geq 0$, we obtain that 
    $[0,\infty)\ni t\mapsto\norm{CS_t x}_Y$ is lower semi-continuous and hence measurable.
\end{proof}

In view of Lemma \ref{lem:measurability}, we can apply Theorem \ref{thm:spectral+diss-obs} to obtain the following.

\begin{theorem} \label{thm:spectral+diss-obs_bi-continuous} 
Let $(X,\norm{\cdot}_X,\topo_X)$ be a sequentially complete Saks space, $(S_t)_{t\geq 0}$ a locally $\topo_X$-bi-continuous semigroup on $X$, $(Y,\norm{\cdot}_Y,\topo_Y)$ a Saks space, $C\in\calL(X;Y)$ such that $C$ is also sequentially $\topo_X$-$\topo_Y$-continuous 
on $\norm{\cdot}_X$-bounded sets.
Further, let $\lambda^* \geq 0$, $(P_\lambda)_{\lambda>\lambda^*}$ a family in $\calL(X)$,
$r \in [1,\infty]$, $d_0,d_1,d_3,\gamma_1,\gamma_2,\gamma_3,T > 0$ with $\gamma_1 < \gamma_2$, and $d_2\geq 1$, and assume that
\begin{equation} \label{eq:UP}
\forall x\in X \ \forall \lambda > \lambda^* \colon \quad 
\norm{P_\lambda x}_{ X } \le d_0 \euler^{d_1 \lambda^{\gamma_1}} \norm{C  P_\lambda x}_{Y }
\tag{UP'}
\end{equation}
and
\begin{equation}\label{eq:DISS}
\forall x\in X \ \forall \lambda > \lambda^* \ \forall t\in (0,T/2] \colon \quad \norm{(I-P_\lambda) S_t x}_{X} \le d_2 \euler^{-d_3 \lambda^{\gamma_2} t^{\gamma_3}} \norm{x}_{X}.
\tag{DISS'}
\end{equation}
Then there exists $C_{\mathrm{obs}}\geq 0$ such that for all $x \in X$ we have
\begin{equation*} 
 \norm{S_T x }_{X} \leq \begin{cases}
    C_{\mathrm{obs}} \left(\int_0^T \norm{CS_t x}_Y^r \drm t\right)^{1/r} & \text{if } r\in [1,\infty),\\
    C_{\mathrm{obs}}\esssup_{t\in [0,T]} \norm{CS_t x}_Y & \text{if } r=\infty .
 \end{cases} 
\end{equation*}
\end{theorem}

\begin{remark}
    The statements in Theorem \ref{thm:spectral+diss-obs} and Theorem \ref{thm:spectral+diss-obs_bi-continuous} can be generalised in the sense that one can obtain an estimate with an $L_r$-norm of $t\mapsto \norm{CS_tx}_Y$ on a measurable subset $E\subseteq [0,T]$ with positive Lebesgue measure; cf.\ e.g.\ \cite{BGST2023}. However, in this case the constant $C_{\mathrm{obs}}$ (cf.\ Remark \ref{rem:constant}) is not explicit anymore.
\end{remark}

\section{Two Examples of Bi-Continuous Semigroups}
\label{sec:application}

In this section we consider final state observability for two important examples: the Gau{\ss}--Weierstra{\ss} semigroup on $C_b(\R^d)$ and the Ornstein--Uhlenbeck semigroup on $C_b(\R^d)$. We begin with the study of restriction operators on $C_b(\R^d)$, restricting functions to suitable subsets, and relate this to an abstract uncertainty principle.

\subsection{Restriction Operators on \texorpdfstring{$C_b(\R^d)$}{Cb(Rd)} and the Uncertainty Principle.}
\label{subsec:Restriction}

Let $\Omega\subseteq \R^d$ be non-empty, $C\colon C_b(\R^d)\to C_b(\Omega)$ the restriction operator defined by $Cf\coloneqq f|_\Omega$ for $f\in C_b(\R^d)$. Then $C\in \calL(C_b(\R^d);C_b(\Omega))$. Let $\topo_{\co}$ be the compact-open topology on $C_b(\R^d)$ (as well as on $C_b(\Omega)$).
Then $(C_b(\Omega),\norm{\cdot}_\infty,\topo_{\co})$ is a Saks space which is sequentially complete if $\Omega$ is locally compact (in particular if $\Omega=\R^d$).

\begin{lemma}
\label{lem:restriction_continuous}
    $C\from (C_b(\R^d),\topo_{\co})\to (C_b(\Omega),\topo_{\co})$ is continuous.
\end{lemma}

\begin{proof}
The map $C$ is clearly linear. Due to \cite[Theorem 46.8]{munkres2000} the compact-open topology $\topo_{\co}$ 
on $C_b(Z)$ for a Hausdorff topological space $Z$ is given by the system of seminorms 
\[
p_{K}^{Z}(f)\coloneqq\sup_{x\in K}|f(x)|\quad (f\in C_b(Z))
\]
for compact $K\subseteq Z$. Let $K\subseteq\Omega$ be compact in the relative topology. 
Then $K$ is compact in $\R^{d}$ as well and 
\[
p_{K}^{\Omega}(Cf)=\sup_{x\in K}|f|_\Omega(x)|=\sup_{x\in K}|f(x)|=p_{K}^{\R^{d}}(f)\quad (f\in C_b(\R^{d})),
\]
which means that $C$ is continuous.
\end{proof}

We now use the operator $C$ to provide an uncertainty principle based on the well-known Logvinenko--Sereda theorem.
Let $\eta\in C_c[0,\infty)$, $\1_{[0,1/2]} \leq \eta \leq \1_{[0,1]}$.
For $\lambda>0$ let $\chi_\lambda\from \R^d\to\R$, $\chi_\lambda\coloneqq \eta(\abs{\cdot}/\lambda)$, and $P_\lambda\in \calL(C_b(\R^d))$ be defined by $P_\lambda f\coloneqq (\calF^{-1}\chi_\lambda) \ast f$, where $\calF$ denotes the Fourier transformation. By Young's inequality and scaling properties of the Fourier transformation, we have
\[\norm{P_\lambda} \leq \norm{\calF^{-1}\chi_\lambda}_{L_1(\R^d)} = \norm{\calF^{-1}\chi_1}_{L_1(\R^d)} \quad(\lambda>0).\]
Note that for all $f\in C_b(\R^d)$ and $\lambda>0$ we have $\calF P_\lambda f = \chi_\lambda \calF f$ and therefore $\spt \calF P_\lambda f \subseteq B[0,\lambda] \subseteq [-\lambda,\lambda]^d$, where $B[0,\lambda]\coloneqq\{x\in\R^d \mid \abs{x}\leq \lambda\}$ is the closed ball around $0$ with radius $\lambda$.

\begin{definition}
    Let $\Omega\subseteq\R^d$. Then $\Omega$ is called \emph{thick} if $\Omega$ is measurable and there exist $L\in (0,\infty)^d$ and $\rho\in (0,1]$ such that
    \[\lambda^d(\Omega\cap (x+(0,L))) \geq \rho \lambda^d((0,L)) \quad(x\in\R^d),\]
    where $\lambda^d$ denotes the $d$-dimensional Lebesgue measure, and $(0,L) \coloneqq \prod_{j=1}^d (0,L_j)$ is the hypercube with sidelengths contained in $L$.
\end{definition}
Thus, a measurable set $\Omega\subseteq\R^d$ is thick (with parameters $L$ and $\rho$) provided the portion of $\Omega$ in every hypercube with sidelengths contained in $L$ is at least $\rho$.

By the Logvinenko--Sereda theorem (see \cite{LogvinenkoS-74,Kovrijkine-01}), if $\Omega\subseteq \R^d$ is a thick set, then there exist $d_0,d_1>0$ such that
\begin{equation}
\label{eq:UP_LS}
\norm{P_\lambda f}_{C_b(\R^d)} \leq d_0\e^{d_1 \lambda} \norm{CP_\lambda f}_{C_b(\Omega)} \quad(\lambda>0, f\in C_b(\R^d)).
\tag{LS}
\end{equation}
Thus, \eqref{eq:UP_LS} yields an estimate of the form \eqref{eq:UP}.

\subsection{The Gau{\ss}--Weierstra{\ss} Semigroup on \texorpdfstring{$C_b(\R^d)$}{Cb(Rd)}.}
\label{subsec:Gauss-Weierstrass}

Let $k\from (0,\infty)\times\R^d\to\R$ be given by
\[k_t(x)\coloneqq k(t,x) \coloneqq \frac{1}{(4\pi t)^{d/2}} \e^{-\abs{x}^2/(4t)} \quad(t>0, x\in \R^d),\]
the so-called \emph{Gau{\ss}--Weierstra{\ss} kernel}.
For $t\geq 0$ we define $S_t \in \calL(C_b(\R^d))$ by 
\[S_t f \coloneqq \begin{cases} f & t=0,\\ k_t * f & t>0.
                  \end{cases}\]
Note that by Young's inequality and the fact that $\norm{k_t}_{L_1(\R^d)} = 1$ for all $t>0$ we have $\norm{S_tf}_{C_b(\R^d)} \leq \norm{f}_{C_b(\R^d)}$ for $f\in C_b(\R^d)$ and $t\geq 0$. It is easy to see that $(S_t)_{t\geq 0}$ is a semigroup, which is called the \emph{Gau{\ss}--Weierstra{\ss} semigroup}.
Let $\topo_{\co}$ be the compact-open topology on $C_b(\R^d)$. Then $(S_t)_{t\geq0}$ is locally $\topo_{\co}$-bi-continuous; 
see e.g.\ \cite[Examples 6 (a)]{kuehnemund2003}.

For $\lambda>0$ let $P_\lambda \in \calL(C_b(\R^d))$ be defined as in \prettyref{subsec:Restriction}.
By \cite[Proposition 3.2]{BGST2022}, there exist $d_2\geq 1$ and $d_3>0$ such that
\begin{equation}
\label{eq:DISS_GW}
\norm{(I-P_\lambda) S_tf}_{C_b(\R^d)} \le d_2 \euler^{-d_3 \lambda^{2} t} \norm{f}_{C_b(\R^d)} \quad (\lambda>0, t\geq 0, f\in C_b(\R^d)),
\tag{DISS(GW)}
\end{equation}
i.e.\ a dissipation estimate \eqref{eq:DISS} is fulfilled.

Thus, if $\Omega\subseteq\R^d$ is thick, then \eqref{eq:UP_LS} and \eqref{eq:DISS_GW} provide the estimates \eqref{eq:UP} and \eqref{eq:DISS} and so
\prettyref{thm:spectral+diss-obs_bi-continuous} yields a final state observability estimate for the Gau{\ss}--Weierstra{\ss} semigroup on $C_b(\R^d)$.

\subsection{The Ornstein--Uhlenbeck Semigroup on \texorpdfstring{$C_b(\R^d)$}{Cb(Rd)}.}
\label{subsec:Ornstein--Uhlenbeck}

Let $M\from (0,\infty)\times  \R^d\times\R^d\to \R$ be given by
\[M_t(x,y)\coloneqq M(t,x,y)\coloneqq \frac{1}{\pi^{d/2} (1-\e^{-2t})^{d/2}} \e^{-\abs{y-\e^{-t}x}^2/(1-\e^{-2t})} \quad(t>0, x,y\in\R^d),\]
the so-called \emph{Mehler kernel}.
For $t\geq 0$ we define $S_t \in \calL(C_b(\R^d))$ by 
\[S_t f \coloneqq \begin{cases} f & t=0,\\ \int_{\R^d} M_t(\cdot,y)f(y)\,\drm y & t>0.
                  \end{cases}\]
Since $\int_{\R^d} M_t(\cdot,y)\,\drm y = 1$ for all $t>0$, we have $\norm{S_tf}_{C_b(\R^d)} \leq \norm{f}_{C_b(\R^d)}$ for $f\in C_b(\R^d)$ and $t\geq 0$. It is not difficult to see that $(S_t)_{t\geq 0}$ is a semigroup, which is called the \emph{Ornstein--Uhlenbeck semigroup}. Let $\topo_{\co}$ be the compact-open topology on $C_b(\R^d)$. 
Then $(S_t)_{t\geq0}$ is locally $\topo_{\co}$-bi-continuous on $C_b(\R^d)$; see e.g.\ \cite[Proposition 3.10]{kuehnemund2001}.

Define $k\from (0,1)\times \R^d\to\R$ by
\[k_s(x)\coloneqq \frac{1}{\pi^{d/2} (1-s^2)^{d/2}} \e^{-\abs{x}^2/(1-s^2)}.\]
Let $s\in (0,1)$. Then we obtain
\[M_{\ln\frac{1}{s}} (\tfrac{1}{s}x,y) = \frac{1}{\pi^{d/2} (1-s^2)^{d/2}} \e^{-\abs{y-x}^2/(1-s^2)} = k_s(x-y) \quad(x,y\in\R^d).\]
Hence, 
\[\bigl(S_{\ln\frac{1}{s}} f\bigr)(\tfrac{1}{s}\,\cdot\,) = k_s \ast f \quad(f\in C_b(\R^d)).\]

For $\lambda>0$ let $P_\lambda\in \calL(C_b(\R^d))$ as in Subsection \ref{subsec:Gauss-Weierstrass}.
Since
\[\calF k_s (\xi) = \e^{-(1-s^2)\abs{\xi}^2/4}\eqqcolon h_s(\xi) \quad(\xi\in\R^d),\]
 for $\lambda>0$ and $s\in (0,1)$ we conclude that
\[\bigl((I-P_\lambda)S_{\ln\frac{1}{s}}f\bigr)(\tfrac{1}{s}\,\cdot\,) = (I-P_{\lambda/s}) \bigl(S_{\ln\frac{1}{s}}f(\tfrac{1}{s}\,\cdot\,)\bigr) = \calF^{-1}\bigl((1-\chi_{\lambda/s}) h_s\bigr) \ast f \quad(f\in C_b(\R^d)).\]

\begin{lemma}
\label{lem:DISS_Ornstein-Uhlenbeck}
    There exist $d_2\geq 1$ and $d_3>0$ such that for $\lambda>0$ and $s\in (0,1)$ we have
    \[\norm{\calF^{-1}\bigl((1-\chi_\lambda) h_s\bigr)}_{L_1(\R^d)} \leq d_2 \e^{-d_3 \lambda^2 (1-s^2)}.\]
\end{lemma}

\begin{proof}
    Let $\lambda>0$, $s\in (0,1)$, and define 
    \[\sigma_{s,\lambda}\coloneqq (1-\chi_{\sqrt{1-s^2} \lambda})h_s\Bigl(\tfrac{1}{\sqrt{1-s^2}}\,\cdot\,\Bigr) = (1-\chi_{\sqrt{1-s^2} \lambda}) \e^{-\abs{\cdot}^2/4}.\]
    Then by a linear substitution we obtain
    \[\norm{\calF^{-1}\bigl((1-\chi_\lambda) h_s\bigr)}_{L_1(\R^d)} = \norm{\calF^{-1}\sigma_{s,\lambda}}_{L_1(\R^d)}.\]
    
    Let $\alpha\in\N_0^d$, $\abs{\alpha}\leq d+1$. Then
    \[\abs{\partial^\alpha \sigma_{s,\lambda}} \leq \1_{\{\abs{\cdot}\geq \sqrt{1-s^2}\lambda/2\}} \abs{\partial^\alpha \e^{-\abs{\cdot}^2/4}} + \sum_{\beta\in \N_0^d, \beta<\alpha} \binom{\alpha}{\beta} \abs{\partial^{\alpha-\beta} (1-\chi_{\sqrt{1-s^2}\lambda})} \abs{\partial^\beta \e^{-\abs{\cdot}^2/4}}.\]

    There exists $K\geq 0$ such that for all $\beta\in\N_0^d$ with $\beta\leq \alpha$ and all $\xi\in\R^d$ we have
    \[\abs{\partial^\beta \e^{-\abs{\cdot}^2/4}(\xi)} \leq K (1+\abs{\xi})^{\abs{\beta}} \e^{-\abs{\xi}^2/4}.\]
    Let
    \[C_1\coloneqq  \sup_{\beta\in\N_0^d, \beta\leq \alpha, \xi\in\R^d} K (1+ \abs{\xi})^{\abs{\beta}} \e^{-\abs{\xi}^2/16}.\]
    Then, for $\beta\leq \alpha$ and $\abs{\xi}\geq \sqrt{1-s^2}\lambda/2$ we have
    \[\abs{\partial^\beta \e^{-\abs{\cdot}^2/4}(\xi)} \leq C_1 \e^{-\abs{\xi}^2/16} \e^{-(1-s^2)\lambda^2/32}.\]
    Further, for $\beta<\alpha$ and $\xi\in\R^d$ we have
    \begin{align*}
        & \abs{\partial^{\alpha-\beta} (1-\chi_{\sqrt{1-s^2}\lambda})(\xi)} \\
        & \leq (\sqrt{1-s^2}\lambda)^{-\abs{\alpha-\beta}} \abs{\partial^{\alpha-\beta} \chi_1\Bigl(\tfrac{\xi}{\sqrt{1-s^2}\lambda}\Bigr)}  \1_{\{\sqrt{1-s^2}\lambda/2 \leq \abs{\cdot}\leq \sqrt{1-s^2}\lambda\}}(\xi) \\
        & \leq C_2 (\sqrt{1-s^2}\lambda)^{-\abs{\alpha-\beta}} \1_{\{\sqrt{1-s^2}\lambda/2 \leq \abs{\cdot}\leq \sqrt{1-s^2}\lambda\}}(\xi)
    \end{align*}
    where $C_2\coloneqq \max_{\beta<\alpha} \norm{\partial^{\alpha-\beta} \chi_1}_{C_b(\R^d)}$.
    Hence, there exists $C\geq0$ (which is independent of $s$ and $\lambda$) such that if $\sqrt{1-s^2}\lambda\geq 1$, then for all $\xi\in\R^d$ we have
    \[\abs{\partial^\alpha \sigma_{s,\lambda}} \leq C \e^{-\abs{\xi}^2/16} \e^{-(1-s^2)\lambda^2/32}.\]
    
    Therefore, increasing $C$, for all $x\in\R^d$ we obtain
    \begin{equation}\label{eq:estimate_Fourier_times_pol}
    \abs{x^\alpha \calF^{-1}\sigma_{s,\lambda}(x)}  = \abs{\calF^{-1}(\partial^\alpha \sigma_{s,\lambda})(x)} \leq C\e^{-(1-s^2)\lambda^2/32}.
    \end{equation}
    By choosing $j\in \{1,\ldots,d\}$ and $\alpha \coloneqq (d+1)e_j$ for the $j$-th canonical unit vector $e_j$, we observe
    $\norm{x}_\infty^{d+1} \abs{\calF^{-1}\sigma_{s,\lambda}(x)}  \leq  C\e^{-(1-s^2)\lambda^2/32}$ and hence
    \begin{equation}\label{eq:estimate_Fourier_outside_zero}
    \abs{\calF^{-1}\sigma_{s,\lambda}(x)}  \leq  C\e^{-(1-s^2)\lambda^2/32} \abs{x}^{-d-1}
    \end{equation}
    for all $x\in\R^d\setminus\{0\}$, where we increased $C$.
    
    Therefore, if $\sqrt{1-s^2}\lambda\geq1$, we can conclude by \eqref{eq:estimate_Fourier_times_pol} 
    for $\alpha=0$ and \eqref{eq:estimate_Fourier_outside_zero} that
    \begin{align*}
        \norm{\calF^{-1}\bigl((1-\chi_\lambda) h_s\bigr)}_{L_1(\R^d)} & = \norm{\calF^{-1}\sigma_{s,\lambda}}_{L_1(\R^d)} \\
        & \leq C\e^{-(1-s^2)\lambda^2/32} \Biggl(\;\int_{B[0,1]} 1\,\drm x + \int_{\R^{d}\setminus B[0,1]} \abs{x}^{-d-1}\,\drm x\Biggr) \\
        & \leq C\e^{-(1-s^2)\lambda^2/32},
    \end{align*}
    where we increased $C$ again.
    
    It remains to prove the estimate for the case $\sqrt{1-s^2}\lambda<1$. Note that
    \begin{align*}
        \norm{\calF^{-1} (\chi_\lambda h_s)}_{L_1(\R^d)} & = \norm{ \calF^{-1} \chi_\lambda \ast \calF^{-1} h_s}_{L_1(\R^d)} \\
        & \leq \norm{\calF^{-1} \chi_\lambda}_{L_1(\R^d)} \norm{k_s}_{L_1(\R^d)} = \norm{\calF^{-1} \chi_1}_{L_1(\R^d)},
    \end{align*}
    where the last equality follows form scaling properties of the Fourier transformation and the fact that 
    $k_s$ is normalised in $L_1(\R^d)$.
    
    Thus, for $\sqrt{1-s^2}\lambda<1$ we obtain
    \[\norm{\calF^{-1} (\chi_\lambda h_s)}_{L_1(\R^d)} \leq \norm{\calF^{-1} \chi_1}_{L_1(\R^d)} \e^{1/32} \e^{-(1-s^2)\lambda^2/32},\]
    which ends the proof.    
\end{proof}

In view of Lemma \ref{lem:DISS_Ornstein-Uhlenbeck}, we obtain the dissipation estimate \eqref{eq:DISS} as follows. Note that for $t\geq 0$ we have $\e^{2t}-1\geq 2t$.
Let $t>0$ and $\lambda>0$, and set $s\coloneqq \e^{-t}\in (0,1)$. Then, for $f\in C_b(\R^d)$, Young's inequality and Lemma \ref{lem:DISS_Ornstein-Uhlenbeck} yield
\begin{align}
\label{eq:DISS_OU}
    \norm{(I-P_\lambda)S_{t}f}_{C_b(\R^d)} & = \norm{\bigl((I-P_\lambda)S_{\ln\frac{1}{s}}f\bigr)(\tfrac{1}{s}\,\cdot\,)}_{C_b(\R^d)} \nonumber \\
    & \leq \norm{\calF^{-1}\bigl((1-\chi_{\lambda/s})h_s\bigr)}_{L_1(\R^d)} \norm{f}_{C_b(\R^d)} \nonumber \\
    & \leq d_2 \e^{-d_3 \lambda^2 s^{-2} (1-s^2)}\norm{f}_{C_b(\R^d)} = d_2 \e^{-d_3 \lambda^2 (\e^{2t}-1)}\norm{f}_{C_b(\R^d)} \nonumber\\
    & \leq d_2 \e^{-2d_3 \lambda^2 t}\norm{f}_{C_b(\R^d)}.
    \tag{DISS(OU)}
\end{align}
Thus, if $\Omega\subseteq \R^d$ is thick and $C\from C_b(\R^d)\to C_b(\Omega)$ is the restriction map as in \prettyref{subsec:Gauss-Weierstrass},  then \eqref{eq:UP_LS} and \eqref{eq:DISS_OU} provide the estimates \eqref{eq:UP} and \eqref{eq:DISS} and so \prettyref{thm:spectral+diss-obs_bi-continuous} yields a final state observability estimate for the Ornstein--Uhlenbeck semigroup on $C_b(\R^d)$.

\section{Cost-Uniform Approximate Null-Controllability and Duality}
\label{sec:Null-controllability}

In this section we want to show that cost-uniform approximate null-controllability is equivalent to final state
observability of the dual system, which is known in the setting of norm-strongly continuous semigroups; 
see \cite{Douglas-66,Carja-88,Vieru-05,YuLC-06}.
In the bi-continuous setting this needs a bit of preparation so that we can formulate the corresponding 
definitions. Since we work in Hausdorff locally convex spaces, the choice of the ``correct'' integral may be delicate. We therefore provide the duality for continuous control functions, and thus relate it to a final state observability estimate of the dual system w.r.t.\ a space of vector measures.

\begin{definition}
Let $(X,\norm{\cdot}_X,\topo_X)$ be a Saks space and $T>0$. We set 
\[
  C_{\topo,b}([0,T];X)
\coloneqq \{f\in C([0,T];(X,\topo_X))\ |\ \|f\|_{\infty}\coloneqq \sup_{t\in[0,T]}\norm{f(t)}_X<\infty\}
\]
where $C([0,T];(X,\topo_X))$ is the space of continuous functions from $[0,T]$ to $(X,\topo_X)$.
\end{definition}

\begin{remark}\label{rem:space_bounded_cont}
Let $(X,\norm{\cdot}_X,\topo_X)$ be a Saks space and $T>0$. 
\begin{enumerate}
\item Since the mixed topology $\gamma_X$ coincides with 
$\topo_X$ on $\norm{\cdot}_X$-bounded sets by \cite[Lemma 2.2.1]{wiweger1961}, and a subset of $X$ is $\norm{\cdot}_X$-bounded 
if and only if it is $\gamma_X$-bounded by \cite[Corollary 2.4.1]{wiweger1961} we have 
\[
 C_{\topo,b}([0,T];X)=C_{b}([0,T];(X,\gamma_X))=C([0,T];(X,\gamma_X)).
\] 
We define two topologies on this space. First, the one given by the norm 
\[
\norm{f}_{\infty}\coloneqq \sup_{t\in [0,T]}\norm{f(t)}_X \qquad (f\in C_{\topo,b}([0,T];X)).
\]
Second, the Hausdorff locally convex topology $\gamma_{\infty}$ induced by the directed system of seminorms given by 
\[
p_{\gamma_\infty}(f)\coloneqq \sup_{t\in [0,T]}p_{\gamma_X}(f(t)) \qquad (f\in C_{\topo,b}([0,T];X))
\]
for $p_{\gamma_X}\in\mathcal{P}_{\gamma_X}$ where $\mathcal{P}_{\gamma_X}$ is a directed system of seminorms that 
induces the mixed topology $\gamma_X$. Clearly, $\gamma_{\infty}$ is coarser than the $\norm{\cdot}_{\infty}$-topology. 
Further, $(C_{\topo,b}([0,T];X),\norm{\cdot}_{\infty})$ is a Banach space.
\item We note that a subset $B\subseteq C_{\topo,b}([0,T];X)$ is $\norm{\cdot}_{\infty}$-bounded if and only if 
it is $\gamma_{\infty}$-bounded since a subset of $X$ is $\norm{\cdot}_X$-bounded if and only if it is 
$\gamma_X$-bounded by \cite[2.4.1 Corollary]{wiweger1961}. So $((C_{\topo,b}([0,T];X),\gamma_{\infty})',\topo_b)$ 
is a topological subspace of $C_{\topo,b}([0,T];X)'=((C_{\topo,b}([0,T];X),\norm{\cdot}_{\infty})',\norm{\cdot}_{C_{\topo,b}([0,T];X)'})$ 
where $\topo_{b}$ denotes the topology of uniform convergence on $\gamma_{\infty}$-bounded sets. 
In the following we use the notation 
$\norm{y'}_{(C_{\topo,b}([0,T];X),\gamma_{\infty})'}\coloneqq \norm{y'}_{C_{\topo,b}([0,T];X)'}$ for all 
$y'\in (C_{\topo,b}([0,T];X),\gamma_{\infty})'$.
\end{enumerate}
\end{remark}

\begin{proposition}\label{prop:pettis_integrand}
Let $(X,\norm{\cdot}_X,\topo_X)$ be a sequentially complete Saks space, $(S_{t})_{t\geq 0}$ a locally $\topo_{X}$-bi-continuous semigroup on $X$ and $T>0$.
Let $v\in C_{\topo,b}([0,T];X)$ and set 
\[
 f\colon [0,T]\to X,\; f(t)\coloneqq  S_{T-t}v(t).
\]
Then $f\in C_{\topo,b}([0,T];X)$.
\end{proposition}
\begin{proof}
We denote by $\mathcal{P}_{\topo_X}$ a system of directed seminorms that generates the topology $\topo_X$ on $X$.
Let $(t_{n})_{n\in\N}$ be a sequence in $[0,T]$ that converges to $t\in[0,T]$
and set $x_{n}:=v(t_{n})-v(t)$ for $n\in\N$. The sequence $(x_{n})_{n\in\N}$ is $\norm{\cdot}_X$-bounded and 
$\topo_{X}\text{-}\lim\limits_{n\to\infty} x_n = 0$ due to our assumptions on $v$.
We have for $p\in\mathcal{P}_{\topo_X}$ that
\begin{align*}
  p(f(t_{n})-f(t))
&=p(S_{T-t_{n}}v(t_{n})-S_{T-t}v(t))\\
&\leq p(S_{T-t_{n}}v(t_{n})-S_{{T}-t_{n}}v(t))+p(S_{T-t_{n}}v(t)-S_{T-t}v(t))\\
&\leq p(S_{T-t_{n}}x_{n})+p((S_{T-t_{n}}-S_{T-t})v(t))\\
&\leq \sup_{s\in[0,T]} p(S_{s}x_{n})+p((S_{T-t_{n}}-S_{T-t})v(t)).
\end{align*}
Combining our estimate above with the local bi-equicontinuity and $\topo_X$-strong continuity of 
the semigroup, we deduce that $(f(t_{n}))_{n\in\N}$ converges to $f(t)$ in $(X,\topo_X)$. 
Hence, $f\in C([0,T];(X,\topo_{X}))$. Furthermore, as the semigroup is exponentially bounded, 
there are $M\geq 1$ and $\omega\in\R$ such that for all $t\in[0,T]$
\begin{align*}
      \norm{f(t)}_{X}&=\norm{S_{T-t}v(t)}_{X}
 \leq \norm{S_{T-t}}_{\calL(X)}\norm{v(t)}_{X}
 \leq M\euler^{\omega (T-t)}\norm{v(t)}_{X}\\
&\leq M\euler^{|\omega| T}\norm{v(t)}_{X},
\end{align*}
which yields that $f$ is $\norm{\cdot}_{X}$-bounded on $[0,T]$ because $v([0,T])$ is $\norm{\cdot}_{X}$-bounded.
\end{proof}

\begin{proposition}\label{prop:pettis}
Let $(X,\norm{\cdot}_X,\topo_X)$ be a sequentially complete Saks space, 
$(S_t)_{t\geq 0}$ a locally $\topo_X$-bi-continuous semigroup on $X$, $(U,\norm{\cdot}_U,\topo_U)$ a Saks space, 
and $B\in\calL(U;X)$ such that $B$ is also 
sequentially $\topo_U$-$\topo_X$-continuous on $\norm{\cdot}_{U}$-bounded sets.
Let $T>0$, $u\in C_{\topo,b}([0,T];U)$ and set $f\colon [0,T]\to X$, $f(t)\coloneqq  S_{T-t}Bu(t)$.
Then $f$ is $\topo_X$-Pettis integrable and $\gamma_X$-Pettis integrable and both integrals coincide.
\end{proposition}
\begin{proof}
The statement follows from \cite[2.5 Proposition (a)]{kruse_schwenninger2023} and \prettyref{prop:pettis_integrand} 
because the map $v\colon t\mapsto Bu(t)$ belongs to $C_{\topo,b}([0,T];X)$. 
\end{proof}

Now, we have everything at hand to formulate the definition of cost-uniform approximate null-controllability 
in the bi-continuous setting. Let $(X,\norm{\cdot}_X,\topo_X)$ be a sequentially complete Saks space, 
$(U,\norm{\cdot}_U,\topo_U)$ a Saks space, $(S_t)_{t\geq 0}$ a locally $\topo_X$-bi-continuous semigroup on $X$ with generator $(-A,D(A))$, 
and $B\in\calL(U;X)$ such that $B$ is also sequentially $\topo_U$-$\topo_X$-continuous on $\norm{\cdot}_{U}$-bounded sets, and $T>0$.
We consider the linear control system 
\begin{align}
\begin{split}\label{eq:controlsystem}
    \dot{x}(t) & = -Ax(t) + Bu(t) \quad(t>0),\\
    x(0) & = x_0 \in X,
\end{split}
\tag{ConSys}
\end{align}
where $u\in C_{\topo,b}([0,T];U)$. The function $x$ is called \emph{state function} and $u$ is called \emph{control function}. 
The unique mild solution of \eqref{eq:controlsystem} is given by Duhamel's formula
\[
x(t)=S_{t}x_{0}+\int_{0}^{t} S_{t-r}Bu(r)\d r\qquad (t\in[0,T])
\]
due to \cite[Proposition 11 (a)]{kuehnemund2003} and \prettyref{prop:pettis}.
Let $\calP_{\topo_X}$ be a directed system of seminorms that induces the topology $\topo_X$.

\begin{definition}
We say that \eqref{eq:controlsystem} is \emph{cost-uniform approximately $\topo_X$-null-controllable in time $T$ via $C_{\topo,b}([0,T];U)$} 
if there exists $C \geq 0$ such that for all $x_0\in X$, $\varepsilon>0$ and $p_{\topo_X}\in\calP_{\topo_X}$ 
there exists $u\in C_{\topo,b}([0,T];U)$ with $\norm{u}_{\infty}\leq C \norm{x_0}_X$ such that $p_{\topo_X}(x(T)) \leq \varepsilon$. 
\end{definition}
We note that this definition of cost-uniform approximate $\topo_X$-null-controllability does not depend 
on the choice of $\calP_{\topo_X}$.

\begin{remark}
    We can analogously define the notion of cost-uniform approximate $\gamma_X$-null-controllability in time $T$ via $C_{\topo,b}([0,T];U)$ by using $p_{\gamma_X}\in\calP_{\gamma_X}$ instead of $p_{\topo_X}\in\calP_{\topo_X}$. In view of \prettyref{prop:cost_uniform_approx_null_contr} and \prettyref{rem:gamma_closure} these two notions are equivalent.
\end{remark}

Next, we prepare the definition of final state observability of the dual system where we need to clarify which kind of 
duality we have to use. 

\begin{definition}
Let $(X,\norm{\cdot}_X,\topo_X)$ be a Saks space and $\K=\R$ or $\C$ the scalar field of $X$. 
\begin{enumerate} 
\item We call $(X,\norm{\cdot}_X,\topo_X)$ \emph{C-sequential} if $(X,\gamma_X)$ is C-sequential, i.e.~every convex sequentially 
open subset of $(X,\gamma_X)$ is already open (see \cite[p.~273]{snipes1973}).
\item We call $(X,\norm{\cdot}_X,\topo_X)$ a \emph{Mazur space} if $(X,\gamma_X)$ is a Mazur space, i.e. 
\[
X_{\gamma}'\coloneqq  (X,\gamma_X)'=\{x'\colon X\to \K\ |\ x'\;\text{linear and } \gamma_{X}\text{-sequentially continuous}\}
\]
(see \cite[p.~40]{wilansky1981}).
\end{enumerate}
\end{definition}

Examples of C-sequential Saks spaces can be found in \cite[Example 2.4, Remarks 3.19, 3.20, Corollary 3.23]{kruse_schwenninger2022}.

\begin{remark}\label{rem:C_sequential_Mazur}
Let $(X,\norm{\cdot}_X,\topo_X)$ be a Saks space. 
\begin{enumerate} 
\item If $(X,\norm{\cdot}_X,\topo_X)$ is C-sequential, then it is a Mazur space by \cite[Theorem 7.4]{wilansky1981} 
(cf.~\cite[3.6 Proposition (b)]{kruse_schwenninger2023}). 
\item The space
\[
X^{\circ}\coloneqq  \{x'\in X'\ |\ x'\;\topo_X\text{-sequentially continuous on } \norm{\cdot}_X\text{-bounded sets}\}
\]
is a closed linear subspace of the norm dual $X'$ and hence a Banach space with norm given by
$\norm{x^{\circ}}_{X^{\circ}}\coloneqq  \norm{x^{\circ}}_{X'}$ for $x^{\circ}\in X^{\circ}$ due to \cite[Proposition 2.1]{farkas2011} 
(note that the proof of \cite[Proposition 2.1]{farkas2011} does not use \cite[Hypothesis A (ii)]{farkas2011} 
which is the sequential completeness of $(X,\norm{\cdot}_X,\topo_X)$). 
We have $X^{\circ}=X_{\gamma}'$ if and only if $(X,\gamma_X)$ is a Mazur space by \cite[3.5 Remark]{kruse_schwenninger2023}. 
\end{enumerate}
\end{remark}

Let $(X,\norm{\cdot}_X)$ and $(U,\norm{\cdot}_U)$ be Banach spaces. We recall that the dual 
operator $B'$ of an element $B\in\calL(U;X)$ is defined by $\langle B'x',u\rangle\coloneqq  \langle x',Bu\rangle$ 
for $x'\in X'$ and $u\in U$.

\begin{proposition}\label{prop:adj_mixed}
Let $(X,\norm{\cdot}_X,\topo_X)$ be a sequentially complete C-sequential Saks space and 
$(S_t)_{t\geq 0}$ a locally $\topo_X$-bi-continuous semigroup on $X$. 
Then the operators given by $S^{\circ}_{t}x^{\circ}\coloneqq S_{t}'x^{\circ}$ for $t\geq 0$ and $x^{\circ}\in X^{\circ}$
belong to $\calL(X^{\circ})$ and form a $\tau_{c}(X^{\circ},(X,\norm{\cdot}_{X}))$-strongly continuous, exponentially bounded semigroup on $X^{\circ}$ where $\tau_{c}(X^{\circ},(X,\norm{\cdot}_{X}))$ denotes the topology of uniform convergence on 
$\norm{\cdot}_{X}$-com\-pact sets.
\end{proposition}
\begin{proof}
Since $(X,\gamma_{X})$ is C-sequential, in particular Mazur by \prettyref{rem:C_sequential_Mazur} (a), 
$(S_t)_{t\geq 0}$ is quasi-$\gamma_{X}$-equicontinuous and 
$X^{\circ}=X_{\gamma}'$ by \cite[Theorem 3.17 (a)]{kruse_schwenninger2022} and \prettyref{rem:C_sequential_Mazur} (b). 
In particular, $S^{\circ}_{t}$ is the $\gamma_{X}$-dual map of $S_{t}$ for all $t\geq 0$.
Furthermore, $(S^{\circ}_{t})_{t\geq 0}$ is exponentially bounded (w.r.t.~$\norm{\cdot}_{\calL(X^{\circ})}$) because 
$(S_t)_{t\geq 0}$ is exponentially bounded.
It follows that $S^{\circ}_{t}\in \calL(X^{\circ})$ for all $t\geq 0$ and $(S^{\circ}_{t})_{t\geq 0}$ is a 
$\sigma(X^{\circ},X)$-strongly continuous semigroup. As $\sigma(X^{\circ},X)$ and the mixed topology 
$\gamma^{\circ}\coloneqq \gamma(\|\cdot\|_{X^{\circ}},\sigma(X^{\circ},X))$ coincide on $\|\cdot\|_{X^{\circ}}$-bounded sets 
by \prettyref{defi:saks_mixed} (b), $(S^{\circ}_{t})_{t\geq 0}$ is also $\gamma^{\circ}$-strongly continuous. 
Due to \cite[Proposition 3.22 (a)]{kruse_schwenninger2022} we have $\gamma^{\circ}=\tau_{c}(X^{\circ},(X,\norm{\cdot}_{X}))$. 
\end{proof}

The semigroup $(S^{\circ}_{t})_{t\geq 0}$ in the setting of \prettyref{prop:adj_mixed} resembles a bi-continuous semigroup. 
For instance, we note that the generator $(-A^{\circ}, D(A^{\circ}))$ of $(S^{\circ}_{t})_{t\geq 0}$ from \prettyref{prop:adj_mixed} is given by 
\begin{align*}
D(A^{\circ}) & = \Bigl\{x^{\circ}\in X^{\circ}\ |\ \tau_{c}(X^{\circ},(X,\norm{\cdot}_{X}))
\text{-}\lim_{t\to 0\rlim}\frac{S^{\circ}_{t}x^{\circ}-x^{\circ}}{t}\;\text{exists in } X^{\circ}\Bigr\},\\
-A^{\circ}x & = \tau_{c}(X^{\circ},(X,\norm{\cdot}_{X}))\text{-}\lim_{t\to 0\rlim}\frac{S^{\circ}_{t}x^{\circ}-x^{\circ}}{t}
\qquad (x^{\circ}\in D(A^{\circ}))
\end{align*}
and fulfils 
\begin{align*}
D(A^{\circ}) & = \Bigl\{x^{\circ}\in X^{\circ}\ |\ 
\sigma(X^{\circ},X)\text{-}\lim_{t\to 0\rlim}\frac{S^{\circ}_{t}x^{\circ}-x^{\circ}}{t}\;\text{ex.~in } X^{\circ},\;
\sup_{t\in(0,1]}\frac{\norm{S^{\circ}_{t}x^{\circ}-x^{\circ}}_{X^{\circ}}}{t}<\infty\Bigr\},\\
-A^{\circ}x & = \sigma(X^{\circ},X)\text{-}\lim_{t\to 0\rlim}\frac{S^{\circ}_{t}x^{\circ}-x^{\circ}}{t}
\qquad (x^{\circ}\in D(A^{\circ}))
\end{align*}
by \cite[I.1.10 Proposition]{cooper1978} for the mixed topology $\gamma^{\circ}=\tau_{c}(X^{\circ},(X,\norm{\cdot}_{X}))$ 
and the exponential boundedness of $(S^{\circ}_{t})_{t\geq 0}$ 
(cf.~\cite[p.~6]{kruse_schwenninger2023} in the bi-continuous setting). 
What is missing for bi-continuity are sequential completeness of the Saks space 
$(X^{\circ},\|\cdot\|_{X^{\circ}},\sigma(X^{\circ},X))$ and (local) bi-equicontinuity.

\begin{remark}
\label{rem:adj_bi_cont}
Let $(X,\norm{\cdot}_X,\topo_X)$ be a sequentially complete Saks space and 
$(S_t)_{t\geq 0}$ a $\topo_X$-bi-continuous semigroup on $X$ with generator $(-A,D(A))$.
If 
\begin{enumerate}
\item[(i)] $X^{\circ}\cap \{x'\in X'\ |\ \|x'\|_{X'}\leq 1\}$ is sequentially complete w.r.t.\ $\sigma(X^{\circ},X)$,
\item[(ii)] every $\|\cdot\|_{X'} $-bounded $\sigma(X^{\circ},X)$-null sequence in $X^{\circ}$ 
is $\topo_X$-equicontinuous on $\norm{\cdot}_X$-bounded sets,
\end{enumerate}
(see \cite[Hypothesis B and C]{farkas2011}), then 
$(X^{\circ},\|\cdot\|_{X^{\circ}},\sigma(X^{\circ},X))$ is a sequentially complete Saks space and 
$(S^{\circ}_{t})_{t\geq 0}$ is a locally $\sigma(X^{\circ},X)$-bi-continuous semigroup on $X^{\circ}$ by 
\cite[Proposition 2.4]{farkas2011} with generator 
$(-A^{\circ},D(A^{\circ}))$ fulfilling
\begin{align*}
D(A^{\circ}) & = \{x^{\circ}\in X^{\circ}\ |\ \exists y^{\circ}\in X^{\circ}\ \forall x\in D(A): 
\langle -Ax,x^{\circ} \rangle=\langle x,y^{\circ} \rangle\},\\
 -A^{\circ}x^{\circ}& = y^{\circ} \quad(x^{\circ}\in D(A^{\circ})),
\end{align*}
by \cite[Lemma 1]{budde2021}.
\end{remark}

We refer to \cite[3.9 Example]{kruse_schwenninger2023} for examples of sequentially complete Saks spaces satisfying (i) and (ii) of \prettyref{rem:adj_bi_cont}.

\begin{proposition}\label{prop:dual_B}
Let $(X,\norm{\cdot}_X,\topo_X)$ and $(U,\norm{\cdot}_U,\topo_U)$ be Saks spaces, 
and $B\in\calL(U;X)$ such that $B$ is also sequentially $\topo_U$-$\topo_X$-continuous on $\norm{\cdot}_{U}$-bounded sets.
Then $B^{\circ}\coloneqq  B'|_{X^{\circ}}\in\calL(X^{\circ};U^{\circ})$ and is also 
$\sigma(X^{\circ},X)$-$\sigma(U^{\circ},U)$-continuous.
\end{proposition}
\begin{proof}
Let $x^{\circ}\in X^{\circ}$, $u\in U$ and $(u_{n})_{n\in\N}$ a $\norm{\cdot}_{U}$-bounded sequence in $U$ 
that $\topo_U$-converges to $u$. Since $B\in\calL(U;X)$ and $B$ is also sequentially $\topo_U$-$\topo_X$-continuous 
on $\norm{\cdot}_{U}$-bounded sets, we have 
that $(Bu_{n})_{n\in\N}$ is $\norm{\cdot}_{X}$-bounded and $\topo_X$-convergent to $Bu$. This implies 
\[
 \langle B^{\circ}x^{\circ}, u\rangle
=\langle x^{\circ}, Bu\rangle
=\lim_{n\to\infty}\langle x^{\circ}, Bu_{n}\rangle,
\]
yielding $B^{\circ}x^{\circ}\in U^{\circ}$ and the $\sigma(X^{\circ},X)$-$\sigma(U^{\circ},U)$-continuity of $B^{\circ}$. 
Furthermore, we note that 
\begin{align*}
 \norm{B^{\circ}x^{\circ}}_{U^{\circ}}
&=\norm{B^{\circ}x^{\circ}}_{U'}
 =\sup_{\norm{u}_{U}\leq 1}|\langle B^{\circ}x^{\circ}, u\rangle|
 =\sup_{\norm{u}_{U}\leq 1}|\langle x^{\circ}, Bu\rangle|
 \leq\sup_{\norm{u}_{U}\leq 1}\norm{Bu}_{X}\norm{x^{\circ}}_{X^{\circ}}\\
&=\norm{B}_{\calL(U;X)}\norm{x^{\circ}}_{X^{\circ}}
\end{align*}
and thus $B^{\circ}\in\calL(X^{\circ};U^{\circ})$.
\end{proof}

\begin{remark}\label{rem:B_seq_cont}
Let $(X,\norm{\cdot}_X,\topo_X)$ and $(U,\norm{\cdot}_U,\topo_U)$ be Saks spaces, $B\in\calL(U;X)$. Consider the following assertions:
\begin{enumerate}
\item $B$ is $\topo_U$-$\topo_X$-continuous on $\norm{\cdot}_{U}$-bounded sets.
\item $B$ is sequentially $\topo_U$-$\topo_X$-continuous on $\norm{\cdot}_{U}$-bounded sets.
\end{enumerate}
Then (a)$\Rightarrow$(b) holds. Moreover, if $(U,\norm{\cdot}_U,\topo_U)$ is C-sequential, then (b)$\Rightarrow$(a) holds.
\end{remark}
\begin{proof}
The implication (a)$\Rightarrow$(b) is obviously true. Let assertion (b) hold. It follows from \cite[Theorem 2.3.1]{wiweger1961}
that $B$ is sequentially $\gamma_U$-$\topo_X$-continuous. Hence, (a) holds by \cite[Theorem 7.4]{wilansky1981} 
if $(U,\gamma_U)$ is C-sequential.
\end{proof}

\begin{proposition}\label{prop:cost_uniform_approx_null_contr}
Let $(V,\norm{\cdot}_V)$, $(W,\norm{\cdot}_W)$ be Banach spaces, $(Z,\norm{\cdot}_Z,\topo_Z)$ a Saks space, 
$\calP_{\gamma_Z}$ a directed system of seminorms that induces the topology $\gamma_Z$,
$F\in\calL(V;Z)$ and $G\in\calL(W;Z)$. Then the following assertions are equivalent:
\begin{enumerate}
\item There exists $c_1\geq 0$ such that $\norm{F'z'}_{V'}\leq c_{1}\norm{G'z'}_{W'}$ 
for all $z'\in Z_{\gamma}'$.
\item There exists $c_2\geq 0$ such that 
\[
\{Fv\ |\ v\in V,\,\norm{v}_V \leq 1\}\subseteq\overline{\{Gw\ |\ w\in W,\,\norm{w}_W \leq c_2\}}^{\gamma_Z}.
\]
\item There exists $c_3\geq 0$ such that for all $v\in V$, $\varepsilon >0$ and $p_{\gamma_Z}\in\calP_{\gamma_Z}$ 
there is $w\in W$ with $\norm{w}_W \leq c_3\norm{v}_V$ such that $p_{\gamma_Z}(Fv+Gw)\leq\varepsilon$.
\end{enumerate}
Moreover, we can choose $c_1=c_2=c_3$.
\end{proposition}
\begin{proof}
First, we note that if $M$ and $N$ are convex sets in $Z$, then we have $N\subseteq \overline{M}^{\gamma_Z}$ if and only if 
\begin{equation}\label{eq:supporting_function}
\sup_{z\in M} \Re\langle z',z \rangle \leq \sup_{z\in N} \Re\langle z',z \rangle
\end{equation}
for all $z'\in Z_{\gamma}'$ by \cite[p.~220]{Carja-88}. Second, let $z'\in Z_{\gamma}'$. For every 
$z\in Z$ there is $\lambda_z\in\C$ with $|\lambda_z|\leq 1$ such that 
$|\langle z', z\rangle|=\Re\langle z', \lambda_z z\rangle $. Thus, if $M$ and $N$ are additionally 
circled sets, then \eqref{eq:supporting_function} is equivalent to
\[
\sup_{z\in M} |\langle z',z \rangle| \leq \sup_{z\in N} |\langle z',z \rangle|
\]
for all $z'\in Z_{\gamma}'$. 
Hence, setting $N\coloneqq  \{Fv\ |\ v\in V,\,\norm{v}_V \leq 1\}$, $M\coloneqq \{Gw\ |\ w\in W,\,\norm{w}_W \leq c_2\}$ 
and observing that $N,M\subseteq Z$ are convex, circled sets, we obtain the equivalence of (a) and (b), 
and that we can choose $c_1=c_2$. 

Let us turn to the equivalence of (b) and (c). For $\varepsilon >0$ and $p_{\gamma_Z}\in\calP_{\gamma_Z}$ we set 
$U_{\varepsilon, p_{\gamma_Z}}\coloneqq  \{z\in Z\;|\;p_{\gamma_Z}(z)\leq \varepsilon\}$. For any $M\subseteq Z$ we have 
\begin{equation}\label{eq:mixed_closure}
\overline{M}^{\gamma_Z}=\bigcap_{\varepsilon>0,\,p_{\gamma_Z}\in\calP_{\gamma_Z}}M+U_{\varepsilon, p_{\gamma_Z}}
\end{equation}
(see e.g.~\cite[2.1.4 Proposition]{jarchow1981}). Let assertion (b) hold, $v\in V$ with $v\neq 0$, $\varepsilon >0$ and 
$p_{\gamma_Z}\in\calP_{\gamma_Z}$. Then there are $\widetilde{w}\in W$ with $\norm{\widetilde{w}}_W \leq c_2$ 
and $z\in Z$ with $p_{\gamma_Z}(z)\leq\tfrac{\varepsilon}{\norm{v}_{V}}$ such that 
$F\left(-\tfrac{v}{\norm{v}_{V}}\right)=G\widetilde{w}+z$ by (b) and \eqref{eq:mixed_closure}. From writing 
\[
-Fv=\norm{v}_{V}F\left(\frac{-v}{\norm{v}_{V}}\right)=\norm{v}_{V}(G\widetilde{w}+z)=G(\norm{v}_{V}\widetilde{w})+\norm{v}_{V}z,
\]
setting $w\coloneqq \norm{v}_{V}\widetilde{w}$, and 
using $\norm{w}_W=\norm{v}_{V}\norm{\widetilde{w}}_{W}\leq c_{2}\norm{v}_{V}$ and 
\[
 p_{\gamma_Z}(Fv+Gw)
=p_{\gamma_Z}(-\norm{v}_{V}z)
=\norm{v}_{V}p_{\gamma_Z}(z)
\leq\norm{v}_{V}\frac{\varepsilon}{\norm{v}_{V}}
=\varepsilon,
\]
we conclude that (c) holds (the case $v=0$ is obvious). 

Now, let assertion (c) hold. Let $v\in V$ with $\norm{v}_{V}\leq 1$, $\varepsilon >0$ and $p_{\gamma_Z}\in\calP_{\gamma_Z}$. 
Then there is $\widetilde{w}\in W$ with $\norm{\widetilde{w}}_{W}\leq c_3$ such that 
$p_{\gamma_Z}(Fv+G\widetilde{w})\leq\varepsilon$. Setting $w\coloneqq  -\widetilde{w}$, using 
$\norm{w}=\norm{\widetilde{w}}_{W}\leq c_3$ and $Fv=Gw+Fv+G\widetilde{w}$, we see that (b) holds 
due to \eqref{eq:mixed_closure}. The proof of the equivalence of (b) and (c) also shows that we can choose $c_2=c_3$.
\end{proof}

The proof of the equivalence of (a) and (b) is just an adaptation of the proof of 
\cite[Theorem 2.2, (iii)$\Leftrightarrow$(iv)]{Carja-88}.

\begin{remark}\label{rem:gamma_closure}
Let $(Z,\norm{\cdot}_Z,\topo_Z)$ be a Saks space. Since the mixed topology $\gamma_Z$ coincides with $\topo_Z$ on 
$\norm{\cdot}_Z$-bounded sets, we may equivalently replace the $\gamma_Z$-closure by the $\topo_Z$-closure in 
\prettyref{prop:cost_uniform_approx_null_contr} (b) and thus $\calP_{\gamma_Z}$ in (c) by a directed system of seminorms 
$\calP_{\topo_Z}$ that induces the topology $\topo_Z$, too.
\end{remark}

\begin{proposition}\label{prop:convolution_cont}
Let $(X,\norm{\cdot}_X,\topo_X)$ be a sequentially complete C-sequential Saks space,  
$(U,\norm{\cdot}_U,\topo_U)$ a Saks space, 
$(S_t)_{t\geq 0}$ a locally $\topo_X$-bi-continuous semigroup on $X$ 
and $B\in\calL(U;X)$ such that $B$ is also $\topo_U$-$\topo_X$-continuous on $\norm{\cdot}_{U}$-bounded sets.
\begin{enumerate}
\item $(S_{t}B)_{t\geq 0}$ is quasi-$\gamma_U$-$\gamma_X$-equicontinuous.  
\item Let $T>0$ and $\mathcal{B}^{T}\colon C_{\tau,b}([0,T];U)\to X$ be given by 
\[
\calB^{T}u\coloneqq  \int_{0}^{T}S_{T-t}Bu(t)\d t .
\]
Then $\calB^{T}\in\calL(C_{\tau,b}([0,T];U);X)=\calL((C_{\tau,b}([0,T];U),\norm{\cdot}_{\infty});(X,\norm{\cdot}_{X}))$ 
and $\calB^{T}$ is also $\gamma_{\infty}$-$\gamma_{X}$-continuous.
\item Let ${\calB^{T}}^{\circ}x^{\circ}\coloneqq  {\calB^{T}}'x^{\circ}$ for $x^{\circ}\in X^{\circ}$. 
Then ${\calB^{T}}^{\circ}x^{\circ}\in (C_{\tau,b}([0,T];U),\gamma_{\infty})'$ for all $x^{\circ}\in X^{\circ}$.
\end{enumerate}
\end{proposition}
\begin{proof}
(a) Let $M\subseteq U$ be a $\norm{\cdot}_{U}$-bounded set. 
Then the restriction $B|_M\colon M \to X$ of $B$ to $M$ is $\topo_{U}|_M$-$\topo_X$-continuous. Since $B\in\calL(U;X)$, the set 
$B(M)$ is $\norm{\cdot}_{X}$-bounded. As the mixed topology $\gamma_X$ coincides with $\topo_X$ on $\norm{\cdot}_{X}$-bounded 
sets by \prettyref{defi:saks_mixed} (b), it follows that $B|_M$ is $\topo_{U}|_M$-$\gamma_X$-continuous, yielding 
that $B$ is $\gamma_U$-$\gamma_X$-continuous by \cite[I.1.7 Corollary]{cooper1978}. 
Due to \cite[Theorem 3.17 (a)]{kruse_schwenninger2022} $(S_{t}B)_{t\geq 0}$ is quasi-$\gamma_U$-$\gamma_X$-equicontinuous, proving part (a).

(b) Let $\mathcal{P}_{\gamma_X}$ and $\mathcal{P}_{\gamma_U}$ be directed systems of seminorms that induce the mixed topologies 
$\gamma_X$ and $\gamma_U$, respectively. 
For $p_{\gamma_X}\in\mathcal{P}_{\gamma_X}$ we set $V_{p_{\gamma_X}}\coloneqq \{x\in X\;|\;p_{\gamma_X}(x)<1\}$ and 
denote its polar set by 
$V_{p_{\gamma_X}}^{\circ}\coloneqq  \{x'\in X_{\gamma}'\;|\;\forall\;x\in V_{p_{\gamma_X}}:\;|x'(x)|\leq 1\}$. 
It follows from part (a) that there are $C\geq 0$ and $p_{\gamma_U}\in\mathcal{P}_{\gamma_U}$ such that 
for all $u\in C_{\tau,b}([0,T];U)$ we have
\begin{align}\label{eq:gamma_cont_conv}
  p_{\gamma_X}(\calB^{T}u)
&=\sup_{x'\in V_{p_{\gamma_X}}^{\circ}}\Bigl|\langle x',\int_{0}^{T}S_{T-t}Bu(t)\d t\rangle\Bigr| 
 \leq \sup_{x'\in V_{p_{\gamma_X}}^{\circ}}\int_{0}^{T}|\langle x',S_{T-t}Bu(t)\rangle|\d t \nonumber\\
&\leq T\sup_{x'\in V_{p_{\gamma_X}}^{\circ}}\sup_{t\in[0,T]}|\langle x',S_{T-t}Bu(t)\rangle|
 =T\sup_{t\in[0,T]}p_{\gamma_X}(S_{T-t}Bu(t))\\
&\underset{\mathclap{(a)}}{\leq} C T\sup_{t\in[0,T]}p_{\gamma_U}(u(t))\nonumber
\end{align}
where we used \cite[Proposition 22.14]{meisevogt1997} in the first and second to last equation 
to get from $p_{\gamma_X}$ to $\sup_{x'\in V_{p_{\gamma_X}}^{\circ}}$ and back. 
Thus, $\calB^{T}$ is $\gamma_{\infty}$-$\gamma_{X}$-continuous.
Furthermore, since $X_{\gamma}'$ is norming for $X$ by \cite[Lemma 5.5 (a)]{kruse_meichnser_seifert2018}, 
we may choose $\mathcal{P}_{\gamma_X}$ such that $\norm{x}_X=\sup_{p_{\gamma_X}\in\mathcal{P}_{\gamma_X}}p_{\gamma_X}(x)$ 
for all $x\in X$ by \cite[Remark 2.2 (c)]{kruse_schwenninger2022}. 
Due to our previous estimates and the exponential boundedness of $(S_{t})_{t\geq 0}$ we obtain
\begin{align*}
 \norm{\calB^{T}u}_{X}
&=\sup_{p_{\gamma_X}\in\calP_{\gamma_X}}p_{\gamma_X}(\mathcal{B}^{T}u)
 \underset{\eqref{eq:gamma_cont_conv}}{\leq} 
 T\sup_{p_{\gamma_X}\in\calP_{\gamma_X}}\sup_{t\in[0,T]}p_{\gamma_X}(S_{T-t}Bu(t))\\
&=T\sup_{t\in[0,T]}\norm{S_{T-t}Bu(t)}_{X}
 \leq T\sup_{t\in[0,T]}\norm{S_{T-t}}_{\calL(X)}\norm{Bu(t)}_{X}\\
&\leq TM\euler^{|\omega|T}\norm{B}_{\calL(U;X)}\sup_{t\in[0,T]}\norm{u(t)}_{U},
\end{align*}
yielding $\calB^{T}\in\calL(C_{\tau,b}([0,T];U);X)$.

(c) It follows from $X^{\circ}=X_{\gamma}'$ by \prettyref{rem:C_sequential_Mazur} (a) and part (b) 
that ${\calB^{T}}^{\circ}$ is the dual map of the $\gamma_{\infty}$-$\gamma_X$-continuous map $\calB^{T}$ 
and hence ${\calB^{T}}^{\circ}x^{\circ}\in (C_{\tau,b}([0,T];U),\gamma_{\infty})'$ for all $x^{\circ}\in X^{\circ}$.
\end{proof}

Now, we are ready to write down the dual system of \eqref{eq:controlsystem} and to phrase the kind
of final state observability of this dual system we are seeking for. 
Let $(X,\norm{\cdot}_X,\topo_X)$ be a sequentially complete C-sequential Saks space, 
$(U,\norm{\cdot}_U,\topo_U)$ a Saks space, $(S_t)_{t\geq 0}$ a locally $\topo_X$-bi-continuous semigroup on $X$ with generator $(-A,D(A))$, 
and $B\in\calL(U;X)$ such that $B$ is $\topo_U$-$\topo_X$-continuous on $\norm{\cdot}_{U}$-bounded sets, and $T>0$.
Using \prettyref{prop:adj_mixed} and \prettyref{prop:dual_B}, the dual system of \eqref{eq:controlsystem} is given by 
\begin{align}
\begin{split}\label{eq:obssystem}
    \dot{x}(t) & = -A^{\circ}x(t) \quad(t>0),\\
    y(t) & = B^{\circ}x(t) \quad(t\geq 0),\\
    x(0) & = x_0 \in X^\circ.
\end{split}\tag{ObsSys}
\end{align}

\begin{definition}
We say that \eqref{eq:obssystem} satisfies a \emph{final state observability estimate in $(C_{\topo,b}([0,T];U),\gamma_{\infty})'$} 
if there exists $C_{\mathrm{obs}}\geq 0$ such that 
\[
\norm{S_{T}^{\circ}x^{\circ}}_{X^{\circ}}
\leq C_{\mathrm{obs}}\norm{{\calB^{T}}^{\circ}x^{\circ}}_{(C_{\topo,b}([0,T];U),\gamma_{\infty})'}
\] 
for all $x^{\circ}\in X^{\circ}$.
\end{definition}

We spend the remaining part of this section with proving that cost-uniform approximate $\topo_X$-null-controllability 
in time $T$ via $C_{\topo,b}([0,T];U)$ of \eqref{eq:controlsystem} is 
equivalent to a final state observability estimate of \eqref{eq:obssystem} in 
$(C_{\topo,b}([0,T];U),\gamma_{\infty})'$, and that the latter space is actually a certain space of vector measures.

Let $\Omega$ be a Hausdorff locally compact space, $(U,\vartheta_U)$ a Hausdorff locally convex space and 
$\mathcal{P}_{\vartheta_U}$ a directed system of seminorms that induces $\vartheta_U$.
We denote by $\mathscr{B}(\Omega)$ the Borel $\sigma$-algebra on $\Omega$, by $M(\Omega)$ the space of all bounded complex 
(or real) Borel measures on $\Omega$, and by $M(\Omega;(U,\vartheta_U)')$ the space of all finitely additive 
vector measures $\nu\colon\mathscr{B}(\Omega)\to(U,\vartheta_U)'$, i.e.~$\nu(N_{1}\cup N_{2})=\nu(N_{1})+\nu(N_{2})$ 
for all disjoint $N_{1},N_{2}\in\mathscr{B}(\Omega)$, such that 
\begin{enumerate}
\item[(i)] $\nu(\cdot)u\in M(\Omega)$ for all $u\in U$, and
\item[(ii)] there exist $p\in\mathcal{P}_{\vartheta_U}$ and $C\geq 0$ such that 
\[
\sup_{(\mathcal{N},\mathcal{U}_p)}\bigl|\sum_{(N,u)\in (\mathcal{N},\mathcal{U}_p)}\nu(N)u\bigr|\leq C
\]
where the supremum is taken over all finite partitions $\mathcal{N}$ of $\Omega$ into disjoint Borel sets 
and all finite sets $\mathcal{U}_p$ in $U$ such that $p(u)\leq 1$ for all $u\in\mathcal{U}_p$. 
\end{enumerate}
Let $(U,\norm{\cdot}_U,\topo_U)$ be a Saks space and $T>0$. 
By \cite[Theorem 1]{wells1965}, the compactness of $[0,T]$ and \prettyref{rem:space_bounded_cont} (a) the map 
\[
\Theta_\gamma\colon M([0,T];U_{\gamma}')\to (C_{\tau,b}([0,T];U),\gamma_\infty)',\;
\Theta_\gamma(\nu)(u)\coloneqq \int_{0}^{T}u(t)\d \nu ,
\]
is a linear isomorphism. By the same theorem in combination with \cite[Lemma 4]{wells1965} the map 
\[
\Theta_{\norm{\cdot}}\colon M([0,T];U')\to (C([0,T];(U,\norm{\cdot}_{U})),\norm{\cdot}_\infty)',\;
\Theta_{\norm{\cdot}}(\nu)(u)\coloneqq \int_{0}^{T}u(t)\d \nu ,
\]
is a topological isomorphism w.r.t.~the semivariation norm on $M([0,T];U')$ and the dual norm on 
$(C([0,T];(U,\norm{\cdot}_{U})),\norm{\cdot}_\infty)'$ and 
\[
 \norm{\Theta_{\norm{\cdot}}(\nu)}_{(C([0,T];(U,\norm{\cdot}_{U})),\norm{\cdot}_\infty)'}
=\norm{\nu}_{var}\qquad (\nu\in M([0,T];U'))
\]
where the semivariation norm is given by 
\[
\norm{\nu}_{var}\coloneqq  
\sup_{(\mathcal{N},\mathcal{U}_{\norm{\cdot}_U})}\bigl|\sum_{(N,u)\in (\mathcal{N},\mathcal{U}_{\norm{\cdot}_U})}\nu(N)u\bigr|
\qquad (\nu\in M([0,T];U'))
\]
and the supremum is taken over all $(\mathcal{N},\mathcal{U}_{\norm{\cdot}_U})$ as in (ii) above 
with $p$ replaced by $\norm{\cdot}_U$.
We note that it follows from $U_{\gamma}'$ being a topological subspace of $U'$ (see \cite[I.1.18 Proposition]{cooper1978}) 
and $\gamma_{U}$ being coarser than the $\norm{\cdot}_{U}$-topology, that $M([0,T];U_{\gamma}')$ is a topological 
subspace of $M([0,T];U')$ (if equipped with the relative topology).

\begin{theorem}\label{thm:dual_volterra_semivar}
Let $(X,\norm{\cdot}_X,\topo_X)$ be a sequentially complete C-sequential Saks space,  
$(U,\norm{\cdot}_U,\topo_U)$ a Saks space, 
$(S_t)_{t\geq 0}$ a locally $\topo_X$-bi-continuous semigroup on $X$ 
and $B\in\calL(U;X)$ such that $B$ is also $\topo_U$-$\topo_X$-continuous on $\norm{\cdot}_{U}$-bounded sets. 
For $T>0$ we have $B^{\circ}S_{(\cdot)}^{\circ}x^{\circ}\odot\lambda\in M([0,T];U_{\gamma}')$ and 
${\calB^{T}}^{\circ}x^{\circ}=\Theta_{\gamma}(B^{\circ}S_{T-(\cdot)}^{\circ}x^{\circ}\odot\lambda)$ 
for all $x^{\circ}\in X^{\circ}$ as well as 
\[
\norm{{\calB^{T}}^{\circ}x^{\circ}}_{(C([0,T];(U,\norm{\cdot}_{U})),\norm{\cdot}_\infty)'}=\norm{B^{\circ}S_{(\cdot)}^{\circ}x^{\circ}\odot\lambda}_{var} 
\qquad (x^{\circ}\in X^{\circ})
\]
where
\[
(B^{\circ}S_{(\cdot)}^{\circ}x^{\circ}\odot\lambda)(N)u
\coloneqq  \int_{N}\langle B^{\circ}S_{t}^{\circ}x^{\circ},u\rangle \d t
\qquad (N\in\mathscr{B}([0,T]),\,u\in U)
\] 
for the Lebesgue measure $\lambda$. 
\end{theorem}

\begin{proof}
First, we recall that $X^{\circ}=X_\gamma'$ by \prettyref{rem:C_sequential_Mazur} (a) 
as the C-sequential space $(X,\gamma_X)$ is Mazur. 
Let $x^{\circ}\in X^{\circ}$.
Due to \prettyref{prop:convolution_cont} (b) and \prettyref{prop:adj_mixed} $\calB^{T}$ is 
$\gamma_{\infty}$-$\gamma_X$-continuous and 
\[
  \langle{\calB^{T}}^{\circ}x^{\circ},u\rangle
=\int_{0}^{T}\langle x^{\circ},S_{T-t}Bu(t)\rangle \d t 
=\int_{0}^{T}\langle B^{\circ}S_{T-t}^{\circ}x^{\circ},u(t)\rangle \d t 
\]
for all $u\in C_{\tau,b}([0,T];U)$. 

For $N\in\mathscr{B}([0,T])$ and $u\in U$ we define $(B^{\circ}S_{(\cdot)}^{\circ}x^{\circ}\odot\lambda)(N)u
\coloneqq  \int_{N}\langle B^{\circ}S_{t}^{\circ}x^{\circ},u\rangle \d t$ and show that $B^{\circ}S_{T-(\cdot)}^{\circ}x^{\circ}\odot\lambda \in M([0,T];U_{\gamma}')$. 
By the proof of \prettyref{prop:convolution_cont} (b) 
there are $C\geq 0$ and $p_{\gamma_U}\in\mathcal{P}_{\gamma_U}$ such that 
\[
 |\langle B^{\circ}S_{T-t}^{\circ}x^{\circ},u\rangle|
=|\langle x^{\circ},S_{T-t}Bu\rangle|
\leq Cp_{\gamma_U}(u)
\]
for all $t\in[0,T]$ and all $u\in U$. In combination with the continuity of the map 
$t\mapsto \langle B^{\circ}S_{T-t}^{\circ}x^{\circ},u\rangle$ on $[0,T]$ by \prettyref{prop:adj_mixed} 
and \prettyref{prop:dual_B} this implies that 
$B^{\circ}S_{T-(\cdot)}^{\circ}x^{\circ}\odot\lambda\colon\mathscr{B}([0,T])\to U_{\gamma}'$ is a well-defined 
finitely additive vector measure and that
\[
(B^{\circ}S_{T-(\cdot)}^{\circ}x^{\circ}\odot\lambda)(\boldsymbol{\cdot})u=\int_{(\boldsymbol{\cdot})}\langle B^{\circ}S_{T-t}^{\circ}x^{\circ},u\rangle \d t
 \qquad(u\in U)
\]
belongs to $M(\Omega)$. Let $\mathcal{N}$ be a finite partition of $[0,T]$ into disjoint Borel sets 
and $\mathcal{U}_{p_{\gamma_U}}$ a finite subset of $U$ such that $p_{\gamma_U}(u)\leq 1$ for all $u\in\mathcal{U}_{p_{\gamma_U}}$. 
Then we have 
\begin{align*}
  \bigl|\sum_{(N,u)\in (\mathcal{N},\mathcal{U}_{p_{\gamma_U}})}(B^{\circ}S_{T-(\cdot)}^{\circ}x^{\circ}\odot\lambda)(N)u\bigr|
&=\bigl|\sum_{(N,u)\in (\mathcal{N},\mathcal{U}_{p_{\gamma_U}})}
  \int_{N}\langle B^{\circ}S_{T-t}^{\circ}x^{\circ},u\rangle \d t\bigr|\\
&\leq \sum_{(N,u)\in (\mathcal{N},\mathcal{U}_{p_{\gamma_U}})}\lambda(N)Cp_{\gamma_U}(u)\\
&\leq C\sum_{N\in\mathcal{N}}\lambda(N)
 = C\lambda([0,T])
 = CT,
\end{align*} 
yielding $B^{\circ}S_{T-(\cdot)}^{\circ}x^{\circ}\odot\lambda\in M([0,T];U_{\gamma}')$. 
Analogously, $B^{\circ}S_{(\cdot)}^{\circ}x^{\circ}\odot\lambda\in M([0,T];U_{\gamma}')$.

Finally, since 
\begin{align*}
  \langle{\calB^{T}}^{\circ}x^{\circ},u\rangle
&=\int_{0}^{T}\langle B^{\circ}S_{T-t}^{\circ}x^{\circ},u(t)\rangle \d t 
 =\int_{0}^{T}u(t)\d (B^{\circ}S_{T-(\cdot)}^{\circ}x^{\circ}\odot\lambda)\\
&=\Theta_\gamma(B^{\circ}S_{T-(\cdot)}^{\circ}x^{\circ}\odot\lambda)(u)
\end{align*} 
for all $u\in C_{\tau,b}([0,T];U)$ and 
\[
 \langle{\calB^{T}}^{\circ}x^{\circ},u\rangle
=\Theta_\gamma(B^{\circ}S_{T-(\cdot)}^{\circ}x^{\circ}\odot\lambda)(u)
=\Theta_{\norm{\cdot}}(B^{\circ}S_{T-(\cdot)}^{\circ}x^{\circ}\odot\lambda)(u)
\]
for all $u\in C([0,T];(U,\norm{\cdot}_{U}))$, it holds that 
\begin{align*}
  \norm{{\calB^{T}}^{\circ}x^{\circ}}_{(C([0,T];(U,\norm{\cdot}_{U})),\norm{\cdot}_\infty)'}
&=\norm{\Theta_{\norm{\cdot}}(B^{\circ}S_{T-(\cdot)}^{\circ}x^{\circ}\odot\lambda)}_{(C([0,T];(U,\norm{\cdot}_{U})),\norm{\cdot}_\infty)'}\\
&=\norm{B^{\circ}S_{T-(\cdot)}^{\circ}x^{\circ}\odot\lambda}_{var}.
\end{align*} 
Let $\mathcal{N}$ be a finite partition of $[0,T]$ and $\mathcal{U}_{\norm{\cdot}_U}$ a finite set in $U$ such that $\norm{u}_{U}\leq 1$ for all $u\in \mathcal{U}_{\norm{\cdot}_U}$. Then $T-\mathcal{N}\coloneqq  \{T-N \;|\; N\in\mathcal{N}\}$ is also a finite partition of $[0,T]$, and
\begin{align*}
    \sum_{(N,u)\in (\mathcal{N},\mathcal{U}_{\norm{\cdot}_U})}(B^{\circ}S_{T-(\cdot)}^{\circ}x^{\circ}\odot\lambda)(N)u 
    & = \sum_{(N,u)\in (\mathcal{N},\mathcal{U}_{\norm{\cdot}_U})} \int_{N}\langle B^{\circ}S_{T-t}^{\circ}x^{\circ},u\rangle \d t \\
    & = -\sum_{(T-N,u)\in (T-\mathcal{N},\mathcal{U}_{\norm{\cdot}_U})}\; \int_{T-N}\langle B^{\circ}S_{t}^{\circ}x^{\circ},u\rangle \d t.
\end{align*}
Thus, we obtain
\[\norm{{\calB^{T}}^{\circ}x^{\circ}}_{(C([0,T];(U,\norm{\cdot}_{U})),\norm{\cdot}_\infty)'} = \norm{B^{\circ}S_{T-(\cdot)}^{\circ}x^{\circ}\odot\lambda}_{var} = \norm{B^{\circ}S_{(\cdot)}^{\circ}x^{\circ}\odot\lambda}_{var}. \qedhere\]
\end{proof}


\begin{theorem}\label{thm:dual_final_state}
Let $(X,\norm{\cdot}_X,\topo_X)$ be a sequentially complete C-sequential Saks space,  
$(U,\norm{\cdot}_U,\topo_U)$ a Saks space, 
$(S_t)_{t\geq 0}$ a locally $\topo_X$-bi-continuous semigroup on $X$ 
and $B\in\calL(U;X)$ such that $B$ is also $\topo_U$-$\topo_X$-continuous on $\norm{\cdot}_{U}$-bounded sets, and $T>0$.
Then the following assertions are equivalent:
\begin{enumerate}
\item The system in \eqref{eq:controlsystem} is cost-uniform approximately $\topo_X$-null-controllable 
in time $T$ via $C_{\tau,b}([0,T];U)$. 
\item The system in \eqref{eq:obssystem} satisfies a final state observability estimate in $(C_{\topo,b}([0,T];U),\gamma_{\infty})'$.
\end{enumerate}
If additionally $\topo_U=\topo_{\norm{\cdot}_U}$, then each of the preceding assertions is equivalent to:
\begin{enumerate}
\item[(c)] There exists $C_{\mathrm{obs}}\geq 0$ such that 
\[
\forall\;x^{\circ}\in X^{\circ}:\quad
\norm{S_{T}^{\circ}x^{\circ}}_{X^{\circ}}\leq C_{\mathrm{obs}}\norm{B^{\circ}S_{(\cdot)}^{\circ}x^{\circ}\odot\lambda}_{var}. 
\] 
\end{enumerate}
\end{theorem}
\begin{proof}
This statement follows from the equivalence of (a) and (c) in \prettyref{prop:cost_uniform_approx_null_contr} 
with $V\coloneqq  Z\coloneqq  X$, $W\coloneqq  C_{\tau,b}([0,T];U)$ equipped with $\norm{\cdot}_{\infty}$, $F\coloneqq  S_{T}$ and $G\coloneqq \calB^{T}$ 
in combination with \prettyref{thm:dual_volterra_semivar}, \prettyref{rem:gamma_closure} and $X^{\circ}=X_\gamma'$.
\end{proof}

Even in the setting of Banach spaces, i.e.~$\topo_X=\topo_{\norm{\cdot}_X}$, $\topo_U=\topo_{\norm{\cdot}_U}$, 
where we have $C_{\topo,b}([0,T];U)=C([0,T];(U,\norm{\cdot}_{U}))$ and 
\[
(C_{\topo,b}([0,T];U),\gamma_{\infty})'=(C([0,T];(U,\norm{\cdot}_{U})),\norm{\cdot}_{\infty})'=M([0,T];U')
\] 
as well as $X^{\circ}=X'$, the results of \prettyref{thm:dual_final_state} seem to be new.

\section*{Statements and Declarations}

We thank F.~Gabel for valuable comments and corrections. 
K.~Kruse acknowledges the support by the Deutsche Forschungsgemeinschaft (DFG) within the Research Training 
Group GRK 2583 ``Modeling, Simulation and Optimization of Fluid Dynamic Applications''.
 
\bibliography{biblio_Observability_bi-continuous}
\bibliographystyle{plainnat}
\end{document}